\documentclass[12pt]{article}
\usepackage[a4paper, margin=2.5cm]{geometry}

\usepackage{amsmath}
\usepackage{amssymb}
\usepackage{amsthm}
\usepackage{mathtools}
\usepackage{amsfonts}
\usepackage{tikz-cd}
\usepackage{enumerate}
\usepackage[utf8]{inputenc}
\usepackage[T1]{fontenc}
\usepackage{siunitx}

\usepackage{url}
\usepackage{hyperref}
\usepackage{blindtext}
\usepackage{cleveref}

\usepackage{mathrsfs}
\usepackage{bbm}
\usepackage{bbold}
\usepackage[mathscr]{euscript}
\usepackage{stmaryrd}
\usepackage{indentfirst}
\usepackage{libertine}
\usepackage{calligra}
\usepackage{calrsfs}

\usepackage{appendix}

\newcommand{\basef}{\mathbb{k}}

\newcommand{\base}{\mathrm{S}}
\newcommand{\baseS}{\mathrm{S}}

\newcommand{\QQ}{\mathbb{Q}}

\newcommand{\X}{\mathscr{X}}
\newcommand{\Y}{\mathscr{Y}}

\newcommand{\W}{\mathscr{W}}

\newcommand{\sX}{\mathrm{X}}
\newcommand{\sY}{\mathrm{Y}}
\newcommand{\sZ}{\mathrm{Z}}

\newcommand{\G}{\mathrm{G}}
\renewcommand{\L}{\mathrm{L}}
\renewcommand{\P}{\mathrm{P}}

\newcommand{\Q}{\mathrm{Q}}

\newcommand{\sW}{\mathrm{W}}

\newcommand{\GG}{\mathbb{G}}
\newcommand{\Bs}{\mathrm{B}}

\newcommand{\N}{\mathscr{N}}

\newcommand{\ind}{\mathscr{I}}
\newcommand{\res}{\mathscr{R}}

\newcommand{\Hom}{\operatorname{Hom}}

\newcommand{\iHom}{\underline{\operatorname{Hom}}}

\renewcommand{\unit}{1}

\renewcommand{\H}{\mathrm{H}}

\newcommand{\Mod}{\mathrm{Mod}}

\newcommand{\Stk}{\mathrm{Stk}}
\newcommand{\tqStk}{\mathrm{tqStk}}

\newcommand{\Coh}{\mathrm{Coh}}

\newcommand{\C}{\mathscr{C}}

\renewcommand{\Pr}{\mathrm{Pr}}

\newcommand{\Ind}{\operatorname{Ind}}

\newcommand{\K}{\mathrm{K}}

\renewcommand{\L}{\mathrm{L}}
\newcommand{\st}{\mathrm{st}}
\newcommand{\cusp}{\mathrm{cusp}}

\newcommand{\gm}{\mathrm{gm}}
\newcommand{\proj}{\mathrm{proj}}

\newcommand{\resol}{\mathrm{res}}
\newcommand{\redm}{\mathrm{r}}

\let\svc\c
\DeclareRobustCommand\c{\ifmmode{c}\else\expandafter\svc\fi}

\newcommand{\KGL}{\mathrm{KGL}}
\newcommand{\Pure}{\mathrm{DK}_\mathrm{pure}}
\newcommand{\Nis}{\mathrm{N}}

\newcommand{\DK}{\operatorname{DK}}
\newcommand{\DM}{\operatorname{DM}}
\newcommand{\SH}{\operatorname{SH}}
\newcommand{\KH}{\operatorname{KH}}

\DeclareMathOperator{\Map}{Map}

\newtheorem{prop}{Proposition}[section]
\newtheorem{lemma}[prop]{Lemma}
\newtheorem{thm}[prop]{Theorem}
\newtheorem{cor}[prop]{Corollary}

\newtheorem{theorem}{Theorem}

\theoremstyle{definition}
\newtheorem{defn}[prop]{Definition}
\newtheorem{example}[prop]{Example}

\theoremstyle{remark}
\newtheorem{rem}[prop]{Remark}

\title{Weights for \texorpdfstring{K}{$\K$}-motives on stacks}
\author{Thiago Landim}

\begin{document}

\maketitle

\begin{abstract}
    We construct the Chow weight structure on a full subcategory of the category of \texorpdfstring{$\mathrm{K}$}{K}-motives over a tame quotient stack in characteristic zero as defined by Hoyois. We also prove that in a quite general case, this full subcategory is exactly the category of geometric \texorpdfstring{$\mathrm{K}$}{K}-motives. We apply this to give a partial Springer decomposition in the context of \texorpdfstring{$\mathrm{K}$}{K}-motives. 
\end{abstract}

\tableofcontents

\section*{Introduction}

The existence of a (motivic) $t$-structure on the category of Beilinson motives has been conjectured since their conception. In spite of multiple attempts, the existence of such a $t$-structure was only proven for the category of Tate motives (and stratified versions thereof) and for the category of $n$-motives for $n=0$ and $1$.

In a different direction, Bondarko defined in \cite{bondarko_original} a dual notion of $t$-structure which he called \emph{weight structure} and which he proved existed for Beilinson motives on schemes (see also \cite{hebert} and \cite{bondarko_weights}). Later Bondarko and Ivanov in \cite{cdh_weight} proved the existence of a weight structure on the category of $\mathrm{cdh}$-motives with integral coefficients as defined by Cisinski and Déglise in \cite{cdh}.

A good motivic homotopy theory for stacks was defined in \cite{hoyois_1} by Hoyois. For any $\X$ in a category of well-behaved quotient stacks $\tqStk_\baseS$, he was able to define the category of motivic spectra $\SH(\X)$ with the usual six operations and their compatibilities. This opened up the doors to talk about spectrum-valued invariants, like $\K$-theory, and in \cite{hoyois_2} he showed that homotopy invariant $\K$-theory $\KH$ is represented by a motivic spectrum $\KGL$.

Parallel to this development, our understanding of the geometry of algebraic stacks has greatly improved, especially after \cite{etale_local}, and now some proofs written out for schemes may be adapted to the context of stacks. Recently, Aranha and Chowdhury showed in \cite{Chow-weight-stacks} that there exists a weight structure in the category of (geometric) $\mathrm{cdh}$-motives on algebraic stacks over a characteristic zero field. We are going to adapt their methods to the context of $\K$-motives\footnote{After the work of Bondarko and Lugarev in \cite{weights_k-motives} (see also \cite{weight_filtrations_annala}) there exists a weight structure on the category of $\K$-motives on schemes.}.

First, similarly to \cite{weight_filtrations_annala}, we prove the existence of a weight structure on the full subcategory of \emph{resolvable} $\K$-motives. Second, we develop the main properties of geometric $\K$-motives over a general base, and show that in characteristic zero (or, more generally, under the existence of resolutions of singularities functorial for smooth morphisms), the category of geometric $\K$-motives coincides with the category of resolvable $\K$-motives. Our results may be described by the following theorem.

\begin{theorem}[Theorem \ref{thm:weight_res} and Lemma \ref{lemma:gm_resol}]
    Let $\base$ be a Noetherian quasi-excellent $\QQ$-scheme with finite Krull dimension. The functor $\DK_\gm \colon \tqStk_\base \to \mathrm{Cat}_\infty$ defines a $\mathrm{cdh}$-sheaf stable under $\otimes$, $f^*$ and $f_!$ for any quasi-projective morphism $f$ and admiting a bounded weight structure. Taking the Ind-completion, this defines a six functor formalism endowed with a weight structure for which the functors $f^*$ and $f_!$ are left weight exact and the functors $f_*$ and $f^!$ are right weight exact.
\end{theorem}

In the final section, we apply this result to give a decomposition of the category of $\G$-equivariant $\K$-motives on the nilpotent cone $\N_\G$ of a split reductive group $\G$ over a characteristic zero field.

\subsection*{Notations}

We denote by $\Pr_\st^\L$ the category of presentable stable $\infty$-categories and functors between them preserving colimits. If $\C$ is an $\infty$-category and $x, \, y$ are objects of $\C$, we denote by $\mathrm{Map}_\C(x,y)$ the mapping space between $x$ and $y$. Sometimes we denote by $\Hom_\C(x, y) \coloneqq \pi_0 \Map_\C(x,y)$ the set of (homotopy classes) of morphisms. If $\C$ is presentable and stable, then it is enriched in spectra, and we denote by $\mathrm{Maps}_\C(x,y)$ the mapping spectrum.

\section{Preliminaries}

Let $\base$ be a quasi-compact separated base scheme which will be fixed for the whole paper. In this section, we will recall the theory motivic spectra as defined by Hoyois in \cite{hoyois_1}. An approach following \cite[Appendix A]{khan_ravi} is possible and the only extra steps would be applying induction on the length of the scallop decomposition.

\begin{defn}
A flat finitely presented group scheme $\G$ over $\base$ is said to be a \emph{tame group} if:
\begin{enumerate}[(a)]
    \item $\base$ admits a Nisnevich covering by schemes having the $\G$-resolution property \cite[Definition 2.7]{hoyois_1},
    \item  $\G$ is linearly reductive.
\end{enumerate}
\end{defn}

\begin{rem}
As observed in \cite[Remark 6.3]{etale_local}, if one also asks $\G$ to be separated with affine fibers, then condition (a) is redundant.  
\end{rem}

\begin{example}
    Any \emph{nice group} (i.e. an extension of a finite étale group scheme, of order prime to the characteristics of $\base$, by a group scheme of multiplicative type) is tame. 
\end{example}

\begin{example}
    If $\base$ has characteristic zero, then any reductive group is tame.
\end{example}

\begin{defn}
We define the 2-category $\tqStk_\base$ of \emph{tame quotient stacks} as the full subcategory of $\Stk_\base$ containing finitely presented $\base$-stacks with the resolution property, that are global quotient stacks $[\sX/\G]$, for a tame affine group scheme $\G$, such that the natural morphism $[\sX/\G] \to \Bs \G$ is quasi-projective Nisnevich-locally\footnote{A morphism which is quasi-projective Nisnevich-locally on $\base$ will be called an $\Nis$-quasi-projective morphism.} on $\base$.
\end{defn}

For any tame quotient stack $\X$, Hoyois defined the category of genuine motivic spectra $\SH(\X)$ which satisfies the usual compatibilities \cite[Theorem 1.1]{hoyois_1}. By \cite[Definition 5.1]{hoyois_2}, there exists a genuine motivic spectrum $\KGL_\X \in \SH(\X)$ representing homotopy invariant $\K$-theory $\KH$.

\begin{defn}
The category of $\K$-motives is defined as
\[
    \DK(\X) \coloneqq \Mod_{\KGL_\X}(\SH(\X)).
\]
\end{defn}

As it is the case with motivic spectra, $\K$-motives can be encoded as a functor from the category of $\Nis$-quasi-projective correspondence
\[
    \DK \colon \mathrm{Corr}(\tqStk_\base) \to \Pr_\st^\L.
\]

In other words, there exists functors $f^*$, $f_*$, $f_!$, $f^!$ for $\Nis$-quasi-projective morphisms and the bifunctors $\otimes$ and $\iHom$, and they satisfy the usual compatibilities of base change, localization, projection formulas and homotopy invariance. By Bott periodicity, if $f$ is smooth, then $f^! = f^*$, and therefore the left adjoing of $f^*$ is $f_!$.

\begin{prop}
    For any tame quotient stack $\X$, the pullback $\nu^* \colon \DK(\X) \to \DK(\X_\mathrm{red})$ is an equivalence.
\end{prop}

\begin{proof}
    This is a corollary of the localization formula.
\end{proof}

\begin{prop}
    The functor $\DK \colon \mathrm{tqStk}_\base^{\mathrm{op}} \to \Pr_{\st}^\L$ is a $\mathrm{cdh}$-sheaf.
\end{prop}

\begin{proof}
    This follows from \cite[Theorem 6.24]{hoyois_1} and from the definition of $\DK(\X)$ as a category of modules on $\SH(\X)$. Alternatively, the same proof as in \emph{loc. cit.} works in this context.
\end{proof}

Descent for the $\mathrm{cdh}$ is equivalent to excision with respect to Nisnevich squares and abstract blow up squares. This means that, for any abstract blow up square
\[
\begin{tikzcd}
    \W \ar["k", hook]{r} \ar["q"']{d} & \Y \ar["p"]{d} \\
    \mathscr{Z} \ar["i"', hook]{r} & \X,
\end{tikzcd}
\]
one has
\[
    \DK(\X) = \DK(\Y) \times_{\DK(\W)} \DK(\mathscr{Z}),
\]
and similarly for Nisnevich squares. In other words, for any $M \in \DK(\X)$, one has the following Cartesian square
\[
\begin{tikzcd}
    M \ar{r} \ar{d} & i_*i_* M \ar{d} \\
    p_*p^* M \ar{r} & (pk)_* (pk)^* M.
\end{tikzcd}
\]

Before going to the next section, we remark a simple consequence of the functoriality of the resolution of singularities in characteristic zero.

\begin{lemma}\label{lemma:resolution}
    Fix $\base$ a Noetherian quasi-excellent scheme over $\mathrm{Spec}(\mathbb{Q})$. Let $\sX$ be a finite type reduced scheme over $\base$, and let $\G$ be a finite type smooth scheme over $\base$ acting on $\sX$. There exists a $\G$-equivariant projective birational morphism $\mathscr{F}(\sX) \to \sX$ such that $\mathscr{F}(\sX)$ is regular.
\end{lemma}

\begin{proof}
    By \cite[Theorem 1.2.1]{temkin_desingularization}, there exists a projective birational morphism $\mathscr{F}(\sX) \to \sX$ such that $\mathscr{F}(\sX)$ is regular. Moreover, the functoriality of $\mathscr{F}$ with respect to smooth morphisms implies that, if $a \colon \G \times_\base \sX \to \sX$ denotes the action of $\G$ on $\sX$, then the morphism $\mathscr{F}(a) \colon \G \times_\base \mathscr{F}(\sX) \to \mathscr{F}(\sX)$ determines a $\G$-action on $\mathscr{F}(\sX)$ for which the projective map $\mathscr{F}(\sX) \to \sX$ is $\G$-equivariant.
\end{proof}

\section{Weights on \texorpdfstring{$\K$}{K}-motives}

In this section, we will define a weight structure on the category of \emph{resolvable} $\K$-motives, similarly to \cite{weight_filtrations_annala}. This category, which has an \emph{ad hoc} definition, will be better described if $\base$ is a Noetherian quasi-excellent over $\mathrm{Spec}(\mathbb{Q})$ of finite Krull dimension. 

\begin{defn}
    For any tame quotient stack $\X$, the category of pure $\K$-motives $\Pure(\X) \subset \DK(\X)$ is the smallest additive subcategory generated under retracts by the collection
    \[ 
    \{f_! \unit_\Y \mid f \colon \Y \to \X \textrm{ is projective and } \Y \textrm{ is regular Noetherian}\}.
    \]
    The category of \emph{resolvable} $\K$-motives $\DK_\resol(\X) \subset \DK(\X)$ is the full thick subcategory generated by $\Pure(\X)$, in other words, the full subcategory generated by $\Pure(\X)$ under finite limits, finite colimits and retracts.
\end{defn}

\begin{thm}\label{thm:weight_res}
    For any tame quotient stack $\X$, the category $\DK_\resol(\X)$ admits a bounded weight structure whose heart is $\Pure(\X)$.
\end{thm}

\begin{proof}
    By base change and \cite[Remark 5.7]{hoyois_2}, the mapping spectrum
    \[
    \begin{split}
        \Map(f_! \unit_{\Y}, f'_! \unit_{\Y'}) &= \Map(f'^*f_!\unit_{\Y}, \unit_{\Y'}) \\
        &= \Map(\pi_{2!}\pi_1^*\unit_{\Y}, \unit_{\Y'}) \\
        &= \Map(\unit_{\Y \times_\X \Y'}, \pi_2^! \unit_{\Y'}) \\
        &= \K(\Coh(\Y \times_\X \Y')),
    \end{split}
    \]
    is connective by \cite[Proposition 3.1]{k-g-theory_khan}. Therefore \cite[Remark 2.2.6]{sosnilo_heart} implies the existence of the given weight structure.
\end{proof}

\begin{rem}
    The general case for linearly scalloped stacks follows from an induction argument together with \cite[Theorem 3.9]{k-g-theory_khan}.
\end{rem}

\begin{prop}\label{prop:exactness_prop}
Let $f \colon \X' \to \X$ be an $\Nis$-quasi-projective morphism between tame quotient stacks.
\begin{enumerate}
    \item If $f$ is projective, then $f_!$ is weight exact.
    \item If $f$ is smooth, then $f^*$ is weight exact.
\end{enumerate}
\end{prop}

\begin{proof}
It follows from the definition that $f_!$ is weight exact if $f$ projective and that $f^*$ is weight exact if $f$ is smooth.
\end{proof}

\section{Geometric \texorpdfstring{$\K$}{K}-motives}

In this section, we define the notion of geometric $\K$-motives. It is the same definition as the usual one for $\DM$, except Bott periodicity allows us to ignore the Tate twist.

\begin{defn}
The category of geometric $\K$-motives $\DK_\gm(\X) \subset \DK(\X)$ is the full stable thick subcategory generated by the collection
\[
    \{ f_! \unit_\Y  \mid f \textrm{ is smooth and quasi-projective}\}.
\]
\end{defn}

\begin{rem}
Since $\DK$ satisfies Nisnevich excision, replacing quasi-projective by $\mathrm{N}$-quasi-projective in definition of the collection wouldn't change the category $\DK_\gm(\X)$.
\end{rem}

\begin{prop}
If $f \colon \X' \to \X$ is an $\Nis$-quasi-projective morphism between tame quotient stacks, then $f^*$ sends $\DK_\gm(\X)$ into $\DK_\gm(\X')$.
\end{prop}

\begin{proof}
    This is an immediate consequence of the base change formula.
\end{proof}

\begin{lemma}\label{lemma:push_geom}
    Suppose $\base$ is excellent or has only a finite number of different characteristics. If $f \colon \X' \to \X$ is a quasi-projective morphism between tame quotient stacks, then $f_!$ sends $\DK_\gm(\X')$ into $\DK_\gm(\X)$.
\end{lemma}

\begin{proof}
    Since $f$ is quasi-projective, we may factor it as
    \[
        \X' \xhookrightarrow{j} \mathscr{P} \xrightarrow{p} \X,
    \]
    where $j$ is an open immersion and $p$ is projective. Since $j$ is smooth and quasi-projective, $j_!$ sends $\DK_\gm(\X')$ into $\DK_\gm(\mathscr{P})$. Therefore, we may suppose $f \colon \X' \to \X$ is a projective morphism. 
    
    Since $f$ is projective, there exists a locally free sheaf $\mathscr{E}$ on $\X$ such that we may factor $f$ as
    \[
        \X' \xhookrightarrow{i} \mathbb{P}_\X(\mathscr{E}) \xhookrightarrow{p} \X,
    \]
    where $i$ is a closed immersion and $p$ is smooth and projective. Similarly as before, $p_!$ sends the category $\DK_\gm(\mathbb{P}_\X(\mathscr{E}))$ into $\DK_\gm(\X)$, therefore we may suppose $f \colon \X' \to \X$ is a closed immersion.

    Let $i \colon \X' \to \X$ be a closed immersion and let $f_0 \colon \W_0 \rightarrow \X'$ be a smooth, quasi-projective morphism. Let $\G$ be the tame group such that $\X' = [\sX'/\G]$. By \cite[Corollary 6.2]{etale_local} and Nisnevich excision, we may suppose $\G$ is embeddable, and by Zariski descent, we may suppose $\base$ is affine. By \cite[Proposition 2.20]{hoyois_1}, if $\W_0 = [\mathrm{W}_0/\G]$ there exists a $\G$-affine vector bundle $\pi \colon \widetilde{\mathrm{W}}_0 \to \mathrm{W}_0$ such that $\widetilde{\mathrm{W}}_0$ is affine over $\base$. In particular, by \cite[Remark 2.9]{etale_local}, the stack $\widetilde{\W}_0 = [\widetilde{\mathrm{W}}_0/\G]$ is linearly fundamental as defined by \cite[Definition 2.7]{etale_local}. By \cite[Theorem 1.3(2)]{hoyois_2}, the counit $\pi_! \pi^* \to \mathrm{id}$ is an equivalence and therefore
    \[
        (f \circ \pi)_! \unit_{\widetilde{\W}_0} \cong f_! \pi_! \pi^* \unit_{\W_0} \cong f_! \unit_{\W_0},
    \]
    which reduces the proof of the lemma to the case of a linearly fundamental stack.
    
    Finally, if $f_0 \colon \W_0 \rightarrow \X'$ is a smooth, quasi-projective morphism and $\W_0$ is linearly fundamental. By \cite[Theorem 5.1]{artin_algebraization}\footnote{This theorem just ensure smoothness, but \cite[Proposition 5.3(1)]{etale_local} ensures we can take take $f \colon \W \to \X$ to be affine (and therefore representable), since $\X$ has the resolution property and therefore, affine diagonal.}, there exists a smooth representable morphism $f \colon \W \to \X$ such that $f \vert_{\X'} \cong f_0$. If $\mathscr{U} = \X - \X'$ denotes the open complement, then the diagram
    \[
    \begin{tikzcd}
    \W_0 \ar[hook]{r} \ar{d} & \W \ar{d} & \W_\mathscr{U} \ar[hook']{l} \ar{d} \\
    \X' \ar[hook]{r} & \X & \mathscr{U} \ar[hook']{l}
    \end{tikzcd}
    \]
    gives the localization sequence
    \[
        j_! j^* f_! \unit_\W \to f_!\unit_\W \to i_!i^* f_!\unit_\W,
    \]
    which, by base change, implies $i_!(f_{0!} 1_{\W_0}) \in \DK_\gm(\X)$, completing the proof. 
\end{proof}

\begin{prop}
    If $\X$ is a tame quotient stack and $M, \, N \in \DK_\gm(\X)$ are geometric $\K$-motives, then $M \otimes N \in \DK_\gm(\X)$.
\end{prop}

\begin{proof}
    For a fixed $M \in \DK_\gm(\X)$, the full subcategory of objects $N \in \DK(\X)$ such that $M \otimes N \in \DK_\gm(\X)$ is a thick subcategory and the same is true exchanging $M$ and $N$. In particular, it's enough to prove the assertion for $M = f_! \unit_\Y$ and $N = f'_! \unit_{\Y'}$, where $f \colon \Y \to \X$ and $f \colon \Y' \to \X$ are a smooth, quasi-projective morphisms. In this case, the proposition follows from the projection formula and the base change formula.
\end{proof}

\begin{lemma}
    If $\base$ is excellent or has only a finite number of different characteristics, then functor $\DK_\gm( - )$ is a $\mathrm{cdh}$-sheaf.
\end{lemma}

\begin{proof}
    We must prove that $\DK_\gm$ sends Nisnevich squares to Cartesian squares and abstract blowup squares to Cartesian squares. We already know $\DK$ is a $\mathrm{cdh}$-sheaf, therefore it is enough to verify if the condition of being geometric is $\mathrm{cdh}$-local. We will write the argument only for the abstract blowup, since the Nisnevich case is entirely similar. Let
    \[
    \begin{tikzcd}
    \W \ar[hook, "k"]{r} \ar["q"']{d} & \Y \ar["p"]{d} \\
    \mathscr{Z} \ar[hook, "i"']{r} & \X
    \end{tikzcd}
    \]
    is an abstract blowup diagram, and let $M \in \DK(\X)$. If $M$ is geometric, then clearly $p^* M$ and $i^*M$ are geometric. Moreover, if $p^* M$ and $i^* M$ are geometric, then $(pk)^* M$ is also geometric, and by Lemma \ref{lemma:push_geom}, the $\K$-motives $p_!p^* M$, $i_!i^* M$ and $(pk)_!(pk)^* M$ are geometric. By $\mathrm{cdh}$-excision, this implies $M$ is also geometric.
\end{proof}

\begin{lemma}\label{lemma:equiv_proj_gm}
    Let $\X$ be a tame quotient stack and let $\DK_\proj(\X)$ denote the full thick subcategory of $\DK(\X)$ generated by the collection
    \[
        \{ f_! \unit_\Y \mid f \textrm{ is projective}\}.
    \]
    Then $\DK_\proj(\X) = \DK_\gm(\X)$.
\end{lemma}

\begin{proof}
    If $f \colon \Y \to \X$ is a projective morphism, then there exists a locally free sheaf $\mathscr{E}$ on $\X$ such that $f$ factors through a closed immersion $i \colon \Y \hookrightarrow \mathbb{P}_\X(\mathscr{E})$. If $\mathscr{U}$ denotes the open complement of $\Y$ in $\mathbb{P}_\X(\mathscr{E})$, then both morphisms $\mathbb{P}(\mathscr{E}) \to \X$ and $\mathscr{U} \to \X$ are smooth and quasi-projective, and applying the localization sequence on this open-closed decomposition proves $\DK_\proj(\X) \subset \DK_\gm(\X)$.

    To prove the other inclusion, let $f \colon \Y \to \X$ be a smooth, quasi-projective morphism. By hypothesis, $f \colon \Y \to \X$ factors as
    \[
        \Y \xhookrightarrow{j} \mathscr{P} \xrightarrow{p} \X,
    \]
    where $j$ is an open immersion and $p$ is projective. If $\mathscr{Z}$ denotes the closed complement of $\Y$ inside of $\mathscr{P}$, then both morphisms $\mathscr{Z} \to \X$ and $\mathscr{P} \to \X$ are projective and another application of the localization theorem implies $f_! \unit_\Y$ is in the thick subcategory generated by $g_! \unit_\mathscr{Z}$, where $g \colon \mathscr{Z} \to \X$ is projective.
\end{proof}

\begin{lemma}\label{lemma:gm_resol}
    If $\base$ a Noetherian quasi-excellent scheme over $\mathrm{Spec}(\mathbb{Q})$ of finite Krull dimension, then for any tame quotient stack $\X$ over $\base$, one has $\DK_\gm(\X) = \DK_\resol(\X)$.
\end{lemma}

\begin{proof} 
    In view of Lemma \ref{lemma:equiv_proj_gm}, it suffices to prove that $f_! \unit_\Y \in \DK_\resol(\X)$ for any $f \colon \Y \to \X$ projective morphism. Before beginning the proof, observe that, if $\G$ is a tame group for which $\X = [\sX/\G]$, for any quasi-projective morphism $f \colon \Y \to \X$ there exists a quasi-projective morphism $f \colon Y \to X$ between schemes such that $\Y = [\sY/\G]$. We will prove the theorem on induction on the dimension of the scheme $\sY$. By reduced invariance, we may suppose $\Y = [\sY/\G]$ is reduced. If the $\dim \sY \leq 0$, then $\sY$ is a finite discrete reduced scheme and $\Y$ is, therefore, regular Noetherian (where we use \cite[\href{https://stacks.math.columbia.edu/tag/0DLS}{Tag 0DLS}]{stacks-project}). Suppose now that the result is true for any $\mathscr{Z} = [\sZ/\G]$ such that $\dim \sZ < \dim \sY$.
    
    By Lemma \ref{lemma:resolution}, there exists a projective birational morphisms $p \colon \widetilde{\Y} = [\mathscr{F}(\sY)/\G] \to \Y$ such that $\mathscr{F}(\sY)$ is regular. If $i \colon \mathscr{Z} \to \Y$ denotes the closed immersion complementary to the dense open subscheme where $p$ is an isomorphism, we have a $\mathrm{cdh}$-square 
    \[
    \begin{tikzcd}
    \mathscr{Z}' \ar["q"']{d} \ar["q", hook]{r} & \widetilde{\Y} \ar["p"']{d} \\
    \mathscr{Z} \ar["i", hook]{r} & \Y.
    \end{tikzcd}
    \]

    By $\mathrm{cdh}$-excision, the square
    \[
    \begin{tikzcd}
        \unit_\Y \ar{r} \ar{d} & i_! \unit_{\mathscr{Z}} \ar{d} \\
        p_! \unit_{\widetilde{\Y}} \ar{r} & (pk)_! \unit_{\mathscr{Z}'}
    \end{tikzcd}
    \]
    is Cartesian. By \cite[Lemma 4.20]{Chow-weight-stacks}, the dimensions of $\sZ$ and $\sZ'$ are strictly smaller than the dimension of $\sY$, where $\mathscr{Z}' = [\sZ'/\G]$. By induction hypothesis, applying $f_!$ to the diagram above shows $f_! 1_\Y \in \DK_\resol(\X)$, as we wanted to prove.
\end{proof}

\begin{cor}
    If $\base$ a Noetherian quasi-excellent scheme over $\mathrm{Spec}(\mathbb{Q})$ of finite Krull dimension, then for any tame quotient stack $\X$ over $\base$, there exists a weight structure on $\DK_\gm(\X)$ whose heart is given by $\Pure(\X)$.
\end{cor}

\section{An Application}

Let $\basef$ be an algebraically closed field of characteristic, let $\G$ be a reductive group over $\basef$ and let $\N_\G$ the nilpotent cone of $\G$. In \cite{k-motives}, Eberhardt defined the category of reduced Springer $\K$-motives, in the same way as \cite{reduced_motives} and proved in \cite[Theorem 5.2]{k-motives} that
\[
    \DK_\redm^\mathrm{Spr}(\N_\G/(\G \times \GG_m)) = \mathrm{Perf}(\mathscr{H}_{\mathrm{aff}, \mathbf{q}}),
\]
where $\mathscr{H}_{\mathrm{aff}, \mathbf{q}}$ is the affine Hecke algebra associated to $\G$. He then used this equivalence to relate the left-hand side with the unramified local Langlands correspondence. We prove an orthogonality result which implies, in particular, the category of Springer $\K$-motives admits an orthogonal complement. It may be seen as a strengthening of a $\K$-theoretic version of \cite[Theorem 5.1]{hyperbolic_achar}.

\begin{thm}
    There exists an orthogonal decomposition
    \[
        \DK_\gm(\N_\G/(\G \times \GG_m)) = \bigoplus_\L \DK_\gm(\N_\G/(\G \times \GG_m))_\L
    \]
    indexed by conjugacy classes of Levi subgroups. The same recomposition is also true for \emph{reduced} $\K$-motives.
\end{thm}

\begin{proof}
    Before going into the proof, we will explain the notation, which comes from \cite[\S 3]{gunningham_derived}. For any parabolic subgroup $\P$ with Levi subgroup $\L$, we consider the span
    \[
        \N_\L/(\L \times \GG_m) \xleftarrow{\pi_\P} \N_\P/(\P \times \GG_m) \xrightarrow{\mu_\P} \N_\G/(\G \times \GG_m),
    \]
    where $\mu_\P$ is a smooth map of relative dimension 0 and $\pi_\P$ is a projective map. We define the parabolic induction by
    \[
    \ind_{\L \subset \P}^\G \coloneqq \mu_{\P!} \pi_\P^* \cong \mu_{\P*} \pi_\P^{!} \colon \DK_\gm(\N_\L/(\L \times \GG_m)) \to \DK_\gm(\N_\G/(\G \times \GG_m))
    \]
    and the parabolic restrictions
    \[
    \res_{\L \subset \P}^{\G} \coloneqq \pi_{\P*}\mu_\P^! \colon \DK_\gm(\N_\G/(\G \times \GG_m)) \to \DK_\gm(\N_\L/(\L \times \GG_m))
    \]
    and
    \[
    {}' \res_{\L \subset \P}^{\G} \coloneqq \pi_{\P!}\mu_\P^* \colon \DK_{\gm}(\N_\G/(\G \times \GG_m)) \to \DK_{\gm}(\N_\L/(\L \times \GG_m)).
    \]
    as its right and left adjoints, respectively.

    By Proposition \ref{prop:exactness_prop}, the functor $\ind_{\L \subset \P}^\G$ is weight exact and therefore by \cite[Theorem 1.2.3(9)]{bondarko_weights} the functors $\res_{\L \subset \P}^{\G}$ and ${}' \res_{\L \subset \P}^{\G}$ are right and left adjoints, respectively. By \cite[\S 3]{hyperbolic_stacks}, the motivic coefficient system $\DK$ satisfies hyperbolic localization over algebraic stacks. As already remark in \cite[Proposition 4.3]{parabolic_rest}, this implies that one has an isomorphism $\res_{\L \subset \P}^\G \cong {}'\res_{\L \subset \overline{\P}}^\G$, where $\overline{\P}$ is the parabolic opposite to $\P$. In particular, the functors $\res_{\L \subset \P}^\G$ and ${}'\res_{\L \subset \P}^\G$ are weight exact. We may, therefore, restrict the adjoint functors to the heart of the corresponding categories:
    \[
        \ind_{\L \subset \P}^\G \colon \DK_\gm(\N_\L/(\L \times \GG_m))^{\heartsuit_w} \rightleftharpoons \DK_\gm(\N_\G/(\G \times \GG_m))^{\heartsuit_w} \colon \res_{\L \subset \P}^\G.
    \]
    Let $\L$ be a Levi subgroup of $\G$ and denote by $\DK(\X)^\mathrm{ren} \coloneqq \Ind(\DK_\gm(\X))$ the Ind-completion of the category of geometric motives $\DK_\gm(\X)$ for any tame quotient stack $\X$. 

    \begin{enumerate}[(a)]
        \item We denote by $\DK(\N_\G/(\G \times \GG_m))^\mathrm{ren}_{\geq \L}$ (resp. $\DK(\N_\G/(\G \times \GG_m))^\mathrm{ren}_{> \L}$) the full subcategory of graded sheaves killed by parabolic restriction to any $\L'$ such that $\L' \ngeq \L$ (resp. $\L' \ngtr \L$).
        \item We denote by $\DK(\N_\G/(\G \times \GG_m))^\mathrm{ren}_{\L}$ the quotient $\DK(\N_\G/(\G \times \GG_m))^\mathrm{ren}_{\geq \L}/\DK(\N_\G/(\G \times \GG_m))^\mathrm{ren}_{> \L}$.
        \item We denote by $\DK(\N_\L/(\L \times \GG_m))^\mathrm{ren}_{\cusp}$ the full subcategory of $\DK(\N_\L/(\L \times \GG_m))^\mathrm{ren}$ spanned by cuspidal sheaves.
    \end{enumerate}

    As in \cite[Convention 3.10]{gunningham_derived}, one may identify the quotient $\DK(\N_\G/(\G \times \GG_m))^\mathrm{ren}_{\L}$ with the subcategory of $\DK(\N_\G/(\G \times \GG_m))^\mathrm{ren}$ generated under colimits by $\ind_{\L \subset \P}^\G \DK(\N_\L/(\L \times \GG_m))^\mathrm{ren}_{\cusp}$. Therefore, to prove the orthogonality result, it suffices to show that
    \[
        \Map(\ind_{\L \subset \P}^\G M, \ind_{\L' \subset \P'}^\G M') = 0,
    \]
    for any pair of non-conjugate Levi subgroups $\L$, $\L'$, and for any $M \in \DK_\gm(\N_\L/(\L \times \GG_m))_{\cusp}^{\heartsuit_w}$ and $M' \in \DK_\gm(\N_{\L'}/(\L' \times \GG_m))_{\cusp}^{\heartsuit_w}$. By adjunction, it suffices to understand $\res_{\L \subset \P}^\G \ind_{\L' \subset \P'}^\G$ and we will do so following \cite[\S 2.5]{stacky_springer}. By base change, the functor $\res_{\L \subset \P}^\G \ind_{\L' \subset \P'}^\G$ is given by $!$-pulling and $*$-pushing along the following correspondence
    \[
    \begin{tikzcd}[column sep = -20pt]
        & & \X \ar{dr} \ar{dl} & & \\
        & \N_{\P'}/(\P' \times \GG_m) \ar["\mu_{\P'}"]{dr} \ar["\pi_{\P'}"']{dl} & & \N_{\Q}/(\Q \times \GG_m) \ar["\pi_{\P}"]{dr} \ar["\mu_{\P}"']{dl} & \\
        \N_{\L'}/(\L' \times \GG_m) & & \N_\G/(\G \times \GG_m) & & \N_{\L}/(\L\times \GG_m),
    \end{tikzcd}
    \]
    where $\X$ admits a stratification $\X = \bigsqcup_w \X_w$ indexed by $\sW_{\L'} \backslash \sW / \sW_\L$, by \cite[\S 2.5]{stacky_springer}.

    Similarly, for any representative $w$ of the double coset $\sW_{\L'} \backslash \sW / \sW_\L$, one may take a lift of $w$ in $\G$ as in \cite[Lemma 2.3.1]{stacky_springer} and consider the correspondence
    \[
    \begin{tikzcd}[column sep = -40pt]
        & & \Y_w \ar{dr} \ar{dl} \ar[bend right = 70, "\alpha_w"']{ddll} \ar[bend left = 70, "\beta_w"]{ddrr} & & \\
        & \N_{\P \cap {}^w \L'}/((\P \cap {}^w \L')\times \GG_m) \ar{dr} \ar{dl} & & \N_{\L \cap {}^w \P'}/((\L \cap {}^w \P')\times \GG_m) \ar{dr} \ar{dl} & \\
        \N_{{}^w \L'}/({}^w \L'\times \GG_m) & & \N_{\L \cap {}^w \L'}/((\L \cap {}^w \L')\times \GG_m) & & \N_{\L}/(\L \times \GG_m).
    \end{tikzcd}
    \]

    By \cite[Theorem 2.5.3]{stacky_springer}, there exists a map $f_w \colon \X_w \to \Y_w$ over $\N_{\L'}/(\L'\times \GG_m) \times \N_{\L}/(\L\times \GG_m)$ such that
    \[
        f_{w*} \circ f_w^! \cong \mathrm{id} \colon \DK_\gm(\Y_w) \to \DK_\gm(\Y_w).
    \] 
    Indeed, to prove this, it suffices to explain why \cite[Lemma 2.2.3]{stacky_springer} is still true for $\DK$.

    Since $\DK$ is homotopy invariant, by \cite[Theorem 1.1(8)]{hoyois_1}, it follows that for any vector bundle $f \colon \X \to \Y$, the map $f^*$ is fully faithful, and therefore $f_* f^! \cong f_*f^* \cong \mathrm{id}$. Furthermore, let $\overline{f} \colon \sX/\G \to \sX/\H$ be a map induced by a morphism $f \colon \G \to \H$ exhibiting $\H$ as the reductive quotient of $\G$. Since we are in characteristic zero, there exists a section $\sigma \colon \H \to \G$ and the map $\overline{\sigma} \colon \sX/\H \to \sX/\G$ is a section of $\overline{f}$ and is $\mathrm{R}_\mathrm{u}(\G)$-torsor. Since $\mathrm{R}_\mathrm{u}(\G)$ is also a vector bundle, using once again homotopy invariance, it follows that pullback with respect to $\overline{\sigma}$ is fully faithful and therefore the same is true for $\overline{f}$.
    
    In particular, arguing as in \cite[Proof of Lemma 2.5.1]{stacky_springer}, we have for that $\res_{\L \subset \P}^\G \ind_{\L' \subset \P'}^\G M$ admits a filtration by $\ind_{\L \cap {}^w \L' \subset \L \cap {}^w \P'}^{\L_2} \res_{\L \cap {}^w \L' \subset \L \cap {}^w \P'}^{{}^w \L'} \circ \operatorname{ad}(w^{-1})^* M$ for any $M \in \DK_\gm(\N_{\L'}/(\L' \times \GG_m))$. 

    A fiber sequence between two elements in the heart of a weight structure always splits and therefore, using the fact that parabolic induction and restrictions are weight exact, if $M \in \DK_\gm(\N_{\L'}/(\L' \times \GG_m))^{\heartsuit_w}$ one has
    \[
        \res_{\L \subset \P}^{\G} \ind_{\L' \subset \P'}^{\G} M \cong \bigoplus_{w \in \sW_{\L'} \backslash \sW / \sW_\L} \ind_{\L \cap {}^w \L' \subset \L \cap {}^w \P'}^{\L} \res_{\L \cap {}^w \L' \subset \L \cap {}^w \P'}^{{}^w \L'} \circ \operatorname{ad}(w^{-1})^*M.
    \]

    We can now conclude our proof. Let  $\L$ and $\L'$ be non-conjugate Levi subgroups of $\G$, and let $M \in \DK_\gm(\N_\L/(\L\times \GG_m))_{\cusp}^{\heartsuit_w}$ let $M' \in \DK_\gm(\N_{\L'}/(\L'\times \GG_m))_{\cusp}^{\heartsuit_w}$ be cuspidal $\K$-motives of weight zero. In this case, for any $w \in \G$,
    \[
    \res_{\L \cap {}^w \L' \subset \L \cap {}^w \P'}^{{}^w \L'} \circ \operatorname{ad}(w^{-1})^*M' = 0,
    \]
    since $M'$ is cuspidal and $\L \cap {}^w \L'$ is a proper Levi subgroup of ${}^w \L'$. Therefore
    \[
    \begin{split}
        \Map(\ind_{\L \subset \P}^\G M, \ind_{\L' \subset \P'}^\G M') &= \Map(M, \res_{\L \subset \P}^\G \ind_{\L' \subset \P'}^\G M') \\
        &= \Map(M, 0) \\
        &= 0,
    \end{split}
    \]
    proving the desired orthogonality.
\end{proof}

\end{document}